%% file: main.tex
\documentclass[11pt]{article}
\usepackage[margin=1in]{geometry}
\usepackage{graphicx}
\usepackage{xspace}
\usepackage{xcolor}
\usepackage{amscd}
\usepackage{amsmath}
\usepackage{amssymb}
\usepackage{amstext}
\usepackage{amsthm}
\usepackage{bbold}
\usepackage{bm}
\usepackage{colonequals}
\usepackage{mathtools}
\usepackage{booktabs,tabularx,array}
\usepackage{enumitem}
\usepackage{microtype}
\usepackage{hyperref}

\DeclareSymbolFont{bbold}{U}{bbold}{m}{n}
\DeclareSymbolFontAlphabet{\mathbbold}{bbold}

\newcommand{\F}{\mathbb F}
\newcommand{\cA}{\mathcal A}
\newcommand{\one}{\mathbf 1}
\newcommand{\Tr}{\operatorname{Tr}}
\newcommand{\Spec}{\operatorname{Spec}}

\newcommand{\Cay}{\operatorname{Cay}}

\newtheorem{theorem}{Theorem}[section]
\newtheorem{lemma}[theorem]{Lemma}
\newtheorem{proposition}[theorem]{Proposition}

\theoremstyle{definition}

\newtheorem{question}[theorem]{Question}

\newtheorem{example}[theorem]{Example}
\theoremstyle{remark}
\newtheorem{remark}[theorem]{Remark}

\setlist[itemize]{leftmargin=1.5em,itemsep=0.3em,topsep=0.4em}
\setlist[enumerate]{leftmargin=1.7em,itemsep=0.3em,topsep=0.4em}
\newcolumntype{Y}{>{\raggedright\arraybackslash}X}

\allowdisplaybreaks
\numberwithin{equation}{section}

\title{Abelian Cayley High-Dimensional Expanders with Polylogarithmic Degree}
\author{Songtao Mao\thanks{\texttt{smao13@jhu.edu}, Department of Computer Science, Johns Hopkins University.}}
\date{}

\begin{document}
\maketitle

\begin{abstract}
We construct an explicit infinite family of simple two-dimensional Cayley complexes over $\F_2^n$ whose degree is polynomial in $n$ and whose nontrivial vertex-link eigenvalues lie in $[-\lambda,\lambda]$ for every fixed $\lambda>0$.  For every fixed $d\ge2$, we also obtain an explicit infinite family of weighted $d$-dimensional Cayley complexes over $\F_2^n$ with codimension-two local spectral norm at most $1/d$ and Cayley degree $\Theta_d(n)$. Our two-dimensional construction uses evaluation at rational points of algebraic curves to produce projective direction sets and many functions affine along these directions, which may be useful for further constructions and improvements.
\end{abstract}

\newpage
\setcounter{tocdepth}{2}
\tableofcontents
\newpage

\input{01-introduction}
\input{02-preliminaries}
\input{03-relation-matrices}
\input{04-ag-direction-system}
\input{05-weighted-dimension-lift}
\input{06-open-problems-and-limitations}
\input{Aknowledgement}
\bibliographystyle{alpha}
\bibliography{references}
\appendix
\input{appendix-ai-contributions}
\input{appendix-ag-background}

\end{document}

%% file: 01-introduction.tex
\section{Introduction}
\label{sec:introduction}

High-dimensional expanders (HDXs) generalize expander graphs to simplicial complexes. Their study brings together combinatorial, spectral, and topological notions of expansion, which need not coincide in higher dimensions. We focus on local spectral expansion, whose origins lie in the cohomological criterion of Garland~\cite{Gar73}.

HDXs have applications to face walks~\cite{KM17}, agreement testing~\cite{DK17}, PCPs~\cite{BMVY25}, locally testable and quantum codes~\cite{DEL22,EKZ20,PK22}, and high-rate approximate local list decoding, hardness amplification, and pseudorandom generators~\cite{DHIP26}. Related local-to-global methods give algorithms for approximate counting~\cite{ALGV19} and constraint satisfaction problems on HDXs~\cite{AJT19}. These applications require different expansion assumptions, some stronger than our link spectral bounds.

Several constructions achieve low degree together with strong expansion. Suitable skeleta of Ramanujan complexes give bounded-degree local spectral expanders in every fixed dimension and for arbitrarily small spectral parameters~\cite{LSV05,DK17}. Kaufman and Oppenheim~\cite{KO18} gave a more elementary bounded-degree construction using coset complexes, with one-sided local spectral expansion. More recently, Hopkins and Ray~\cite{HR26} constructed subpolynomial-degree complexes with local spectral, coboundary, and swap-coboundary expansion. A different question is how small the degree can be when one requires abelian Cayley symmetry.

In this paper, we consider abelian Cayley complexes over $\F_2^n$, with $N=2^n$ vertices. For Cayley graphs, Alon and Roichman~\cite{AR94} showed that $O(n)$ random generators suffice for spectral expansion. However, reducing the degree of Cayley HDXs is more difficult, since their links must also satisfy the required spectral bounds. Independently sampling polynomially many uniform generators does not suffice: with high probability, the resulting Cayley graph contains no triangles.

Earlier work already achieved polylogarithmic degree for other notions of higher-dimensional expansion. Conlon~\cite{Con19} constructed Cayley hypergraph expanders over binary vector spaces, and Conlon, Tidor, and Zhao~\cite{CTZ20} extended this approach to every uniformity. Their results concern face walks, discrepancy, and geometric overlap, but do not give arbitrarily small two-sided vertex-link parameters. This distinction matters even in dimension two: a fixed vertex-link bound $\lambda<1/2$ in a connected complex implies a uniform spectral gap in its one-skeleton~\cite{Opp18}.

For arbitrarily small local spectral parameters, Golowich~\cite{Gol23} constructed abelian Cayley local spectral expanders of degree $2^{O(\sqrt n)}$. Dikstein, Liu, and Wigderson~\cite{DLW25} improved this to $2^{n^\varepsilon}$ for every fixed $\varepsilon>0$, in every fixed dimension and for every fixed local spectral parameter. They asked whether the degree can be reduced further to a polynomial in $n$. We address the two-dimensional case of this question.

\begin{question}\cite{DLW25}
\label{ques:polynomial-degree}
For every fixed $\lambda>0$, can we construct two-dimensional Cayley complexes over $\F_2^n$ whose vertex links are two-sided $\lambda$-spectral expanders and whose Cayley degree is bounded by a polynomial in $n$?
\end{question}

Dikstein, Liu, and Wigderson~\cite{DLW25} also noted that their approach cannot attain polynomial degree without a new idea. In this work, we answer Question~\ref{ques:polynomial-degree} affirmatively.

\subsection{Our results}

Our first theorem gives unweighted two-dimensional Cayley complexes with arbitrarily strong two-sided spectral expansion in their vertex links.

\begin{theorem}
\label{thm:main-two-dimensional}
For every fixed $\lambda>0$, there is an explicit infinite family of two-dimensional Cayley complexes over $\F_2^n$, with $n\to\infty$ along the family. Each complex is simple, pure, unweighted, and translation invariant. Its one-skeleton and vertex links are connected, the random-walk operator of every vertex link has all nontrivial eigenvalues in $[-\lambda,\lambda]$, its Cayley degree is $n^{O_\lambda(1)}$, and every edge lies in $O_\lambda(\log n)$ triangles.
\end{theorem}

Our second theorem concerns weighted complexes in arbitrary dimension.

\begin{theorem}
\label{thm:main-higher-dimensional}
For every fixed $d\ge2$, there is an explicit infinite family of weighted $d$-dimensional Cayley complexes over groups $\F_2^n$, with $n\to\infty$ along the family. Each complex is simple, pure, and translation invariant, and its top-face probability distribution has full support. Its one-skeleton and all positive-dimensional links are connected. On every codimension-two link, the random-walk operator has norm at most $1/d$ on the orthogonal complement of the constants. The Cayley degree is $\Theta_d(n)$.
\end{theorem}

Its linear degree is optimal in order: a connected Cayley complex over $\F_2^n$ with connected vertex links has degree at least $2n-1$~\cite{DLW25}. The theorem follows by equipping a skeleton of the construction in~\cite{Gol21} with its induced top-face distribution. It remains open whether, for every fixed $d\ge3$ and $0<\lambda<1/d$, there are explicit families of abelian Cayley complexes over $\F_2^n$ with codimension-two local spectral norm at most $\lambda$ and Cayley degree polynomial in $n$.

\subsection{Organization}

We present the proof of Theorem~\ref{thm:main-two-dimensional} from the outside in, treating later results as black boxes. Section~\ref{sec:preliminaries} proves the theorem from Lemma~\ref{lem:matrix-family}, which supplies relation matrices with large binary kernels and the required real spectral bound. Section~\ref{sec:relation-matrices} proves this lemma using Lemma~\ref{lem:direction-affine-family}, which gives direction sets and many functions affine along these directions. Section~\ref{sec:ag-directions} proves the latter lemma using evaluation on a Garcia--Stichtenoth tower and maximal minors of matrices of linear forms arising from multiplication of functions. Section~\ref{sec:dimension-lift} independently proves Theorem~\ref{thm:main-higher-dimensional} using an induced weighted skeleton of a graph product. Section~\ref{sec:open-problems} lists the open questions. Appendices~\ref{app:ai-research-note} and~\ref{app:ag-background} discuss the AI contributions and research note, and review the algebraic geometry, respectively.

%% file: 02-preliminaries.tex
\section{Preliminaries}
\label{sec:preliminaries}

\subsection{Notation and conventions}

For an integer $m\ge1$, write $[m]=\{1,\ldots,m\}$.  For a prime power $q$, let $\F_q$ denote the field of $q$ elements; for a finite field $F$, define $F^\times=F\setminus\{0\}$.  If $V$ is a finite-dimensional $F$-vector space, then $V^\vee=\operatorname{Hom}_F(V,F)$; spans, kernels, ranks, and row spans are taken over $F$ unless a subscript indicates otherwise.  We write $e_i$ for a standard basis vector, $\mathbf1$ for an all-ones vector, and $d_{\min}(C)$ for the minimum Hamming distance of a binary linear code $C$. For a binary matrix $H$, let $\widetilde H$ denote the real $0/1$ matrix with the same entries.  Thus kernels and row spans of $H$ are taken over $\F_2$, whereas $\widetilde H^{\mathsf T}\widetilde H$ is the real Gram matrix; its $(i,j)$-entry counts the rows in which columns $i$ and $j$ both have entry one.

\subsection{Spectral expansion}

\paragraph{Weighted graphs and two-sided spectral expansion.} Let $G=(V,E_G,w)$ be a finite undirected graph without isolated vertices, with positive symmetric edge weights $w_{uv}=w_{vu}$.  Set $d(u)=\sum_{v:\{u,v\}\in E_G}w_{uv}$ and let $P$ be the random-walk operator, so $P(u,v)=w_{uv}/d(u)$ on edges.  Its stationary measure is $\pi(u)=d(u)/\sum_v d(v)$, and $P$ is self-adjoint for $\langle f,g\rangle_\pi=\sum_{u\in V}\pi(u)f(u)g(u)$. We call $G$ a \emph{two-sided $\lambda$-spectral expander} if it is connected and $\left\|P\bigm|_{\mathbf1^{\perp_\pi}}\right\|_\pi\leq\lambda$. For an unweighted regular graph this is the bound on the absolute value of every nontrivial normalized adjacency eigenvalue.

\paragraph{Weighted complexes and links.} A finite pure $d$-dimensional simplicial complex is a downward-closed collection $X=X(-1)\sqcup X(0)\sqcup\cdots\sqcup X(d)$ whose $i$-faces have cardinality $i+1$, and in which every face is contained in a $d$-face. Here $X(-1)=\{\varnothing\}$, and $X(0)$ consists of the singleton faces, which we identify with the vertices of $X$.  A \emph{weighted simplicial complex} is a pair $(X,\mu_d)$, where $\mu_d$ is a probability distribution with full support on $X(d)$. It induces the lower-dimensional distributions
\[
 \mu_i(\sigma)
 =
 \binom{d+1}{i+1}^{-1}
 \sum_{\substack{\tau\in X(d)\\\sigma\subseteq\tau}}
 \mu_d(\tau).
\]
Throughout, \emph{simple} means that faces occur without multiplicity. For a face $F$, its link is $X_F=\{T\setminus F:T\in X,\ F\subseteq T\}$, equipped with the top-face distribution obtained by conditioning $\mu_d$ on the event $F\subseteq T$. We call $(X,\mu_d)$ \emph{unweighted} if $\mu_d$ is uniform on $X(d)$.

We say that $(X,\mu_d)$ has \emph{two-sided codimension-two local spectral norm at most $\lambda$} if, for every $F\in X(d-2)$, the weighted link graph $X_F$ is connected and
\[
 \left\|P_F\bigm|_{\mathbf1^{\perp_{\pi_F}}}\right\|_{\pi_F}
 \le\lambda,
\]
where $P_F$ and $\pi_F$ are its random-walk operator and stationary measure. This condition applies only to codimension-two links; it does not assert the same spectral bound for links of other dimensions. Connectivity of the one-skeleton and of the remaining positive-dimensional links will be stated separately.

For $d=2$, the codimension-two faces are the vertices, so this is the usual two-sided spectral condition on the vertex links. If the one-skeleton is connected and $\lambda<1/2$, spectral descent bounds its nontrivial eigenvalues by $[-\lambda/(1+\lambda),\lambda/(1-\lambda)]$~\cite{Opp18}.

\subsection{Abelian Cayley complexes over \texorpdfstring{$\mathbb F_2^n$}{F2 to the n}}

Let $W=\mathbb F_2^n$ and let $S\subseteq W\setminus\{0\}$.  The Cayley graph $\operatorname{Cay}(W,S)$ has vertex set $W$ and edge set $\bigl\{\{x,x+s\}:x\in W,\ s\in S\bigr\}$. It is simple because $0\notin S$, and its degree is $|S|$.

The next lemma recalls the Fourier description of the Cayley graph's spectrum.

\begin{lemma}
\label{lem:cayley-character-spectrum}
Assume $S\ne\varnothing$. For $\xi\in W$, define $\chi_\xi(x)=(-1)^{\langle\xi,x\rangle}$. The characters $\{\chi_\xi:\xi\in W\}$ form an orthogonal eigenbasis for the normalized adjacency operator $P_S$ of $\operatorname{Cay}(W,S)$, and
\[
    P_S\chi_\xi
    =\left(\frac1{|S|}\sum_{s\in S}\chi_\xi(s)\right)\chi_\xi.
\]
Consequently, $\operatorname{Cay}(W,S)$ is connected if and only if $S$ spans $W$.
\end{lemma}

\begin{proof}
The eigenvalue identity follows from $\chi_\xi(x+s)=\chi_\xi(x)\chi_\xi(s)$, and the characters form the Fourier basis of $W$. The connected component of $0$ is $\operatorname{span}_{\F_2}S$, which proves the connectivity.
\end{proof}

We call a pure weighted simplicial complex a \emph{Cayley complex} over $W$ if its one-skeleton is $\operatorname{Cay}(W,S)$ and its induced edge weights depend only on the generator label.  It is \emph{translation invariant} if its top-face distribution is invariant under translations by $W$.

An indexed family of translation-invariant Cayley complexes is \emph{explicit} if a deterministic algorithm, given the family index, outputs the ambient dimension $n$, the generators in $\F_2^n$, one representative of each top-face translation orbit, and the exact rational weight of each face in that orbit, in time polynomial in $n$ for fixed parameters. Finite-field elements and rational weights are encoded in binary.

\subsection{Relation matrices}
\label{subsec:relation-matrices}

We use the quotient construction in~\cite{Gol23}. A \emph{relation matrix} is a binary matrix $H\in\F_2^{R\times M}$ whose rows specify linear relations among $M$ generators over $\F_2$. We focus on relation matrices with distinct rows of weight three and constant positive column weight $\rho$. Suppose moreover that $R_H=\operatorname{rowspan}_{\F_2}H$ has minimum distance at least three. Set $\Gamma_H=\F_2^M/R_H$ and $s_i=e_i+R_H$ for $i\in[M]$.  The distance assumption makes the $s_i$ nonzero and pairwise distinct, and they span $\Gamma_H\cong\F_2^{\dim_{\F_2}\ker H}$.  If a row of $H$ has support $\{i,j,k\}$, then $s_i+s_j+s_k=0$.  Define $X_H$ to have vertex set $\Gamma_H$, Cayley generators $S_H=\{s_i:i\in[M]\}$, and, for every such row and every $g\in\Gamma_H$, the triangle $\{g,g+s_i,g+s_i+s_j\}$. Different orderings of $i,j,k$ give the same translation orbit, since $s_i+s_j+s_k=0$. Equip the triangles with the uniform probability distribution.

Two distinct row supports cannot share a pair, since their sum would be a word of weight two in $R_H$.  It follows that $X_H$ is a simple, pure, unweighted, translation-invariant two-dimensional Cayley complex of Cayley degree $M$, and every edge lies in $2\rho$ triangles.  The link at the identity is $2\rho$-regular, with adjacency matrix $\widetilde H^{\mathsf T}\widetilde H-\rho I_M$. Thus its normalized adjacency operator is
\begin{equation}
\label{eq:matrix-link-operator}
 P_H=\frac{\widetilde H^{\mathsf T}\widetilde H-\rho I_M}{2\rho}.
\end{equation}

\begin{lemma}\cite[Lemma~45]{Gol23}
\label{lem:row-weight-three-compiler}
If the vertex link of $0$ in $X_H$ is connected and $\left\|P_H\bigm|_{\mathbf1^\perp}\right\|\leq\lambda$, then $X_H$ has two-sided codimension-two local spectral norm at most $\lambda$.
\end{lemma}

\paragraph{Matrices for the construction.} The following lemma supplies the relation matrices used to prove Theorem~\ref{thm:main-two-dimensional}.

\begin{lemma}
\label{lem:matrix-family}
For every fixed $\lambda>0$, choose $q=2^h\ge8$ with $4/(q-2)<\lambda$. For arbitrarily large integers $t$, one can explicitly construct a binary matrix $H\in\F_2^{R\times M}$, where $M=(q^2-1)q^{2t}$, $R=\rho M/3$, and $\rho=O_\lambda(t)$, with the following properties.
\begin{enumerate}[label=\textup{(\roman*)}]
\item The rows are distinct and have weight three, every column has weight $\rho>0$, and $\operatorname{rowspan}_{\F_2}H$ has minimum distance at least three.
\item Writing $n=\dim_{\F_2}\ker H$, one has
  \begin{equation}
   n\ge 2h\,2^{t/3}.
   \label{eq:matrix-family-nullity}
  \end{equation}
\item The operator $P_H$ in \eqref{eq:matrix-link-operator} is the normalized adjacency of a connected simple regular graph and satisfies
  \[
   \left\|P_H\bigm|_{\one^\perp}\right\|\le\lambda.
  \]
\end{enumerate}
\end{lemma}

The proof is postponed to Subsection~\ref{proof in 2}.

\subsection{Proof of Theorem~\ref{thm:main-two-dimensional}}

\begin{proof}
Apply Lemma~\ref{lem:matrix-family}, fix one of its matrices $H$, and apply the relation-matrix construction above.  Denote the resulting complex by $X$, its group by $\Gamma=\F_2^M/\operatorname{rowspan}_{\F_2}H$, and its generator set by $S$. Then $\dim_{\F_2}\Gamma=n$, the complex is simple, its Cayley degree is $M$, and its vertex-link operator is $P_H$. The generators span $\Gamma$, so its one-skeleton is connected. Lemma~\ref{lem:row-weight-three-compiler} gives the required vertex-link spectral bound. The nullity bound in \eqref{eq:matrix-family-nullity} gives
\begin{equation}
 |S|=M
 \le(q^2-1)\left(\frac{n}{2h}\right)^{6h}
 =n^{O_\lambda(1)}.
 \label{eq:main-degree}
\end{equation}
Choosing the least admissible $q$ gives $h=O(\log(2+1/\lambda))$, so the exponent in \eqref{eq:main-degree} is $O(\log(2+1/\lambda))$. The nullity bound also gives $t=O_\lambda(\log n)$, so the number $2\rho$ of triangles containing each edge is logarithmic.

Let $B\in\F_2^{n\times M}$ have as its rows a basis of $\ker_{\F_2}H$. Then $\ker B=\operatorname{rowspan}_{\F_2}H$, so $v+\operatorname{rowspan}_{\F_2}H\mapsto Bv$ identifies $\Gamma$ with $\F_2^n$. Under this identification, the generators are the columns of $B$, and the uniform measure, degrees, and spectra are unchanged.

Gaussian elimination computes $B$.  Moreover, $R=\rho M/3=n^{O_\lambda(1)}$, and the rows of $H$ give the triangle-orbit representatives.  Every such orbit has size $2^n$: a nonzero translation in an elementary abelian $2$-group cannot stabilize a set of three vertices.  The three edge differences of a triangle are $s_i,s_j,s_k$; since the generators are distinct, they recover the row support. Thus distinct rows give distinct orbits, and every triangle has weight $1/(R2^n)$.  Letting $t$ range over the arbitrarily large values supplied by Lemma~\ref{lem:matrix-family} gives the required explicit infinite family.
\end{proof}

%% file: 03-relation-matrices.tex
\section{Relation matrices from affine restrictions}
\label{sec:relation-matrices}

We prove Lemma~\ref{lem:matrix-family} by constructing a binary relation matrix from the projective directions in Lemma~\ref{lem:direction-affine-family}. Throughout this section, $F=\F_{q^2}$, where $q=2^h\ge8$. We write $\Tr=\Tr_{F/\F_2}$ for the field trace.

\subsection{Functions affine in prescribed directions}
\label{subsec:direction-affine-input}

We identify $PG(t-1,F)=(F^t\setminus\{0\})/F^\times$ and write $[a]=F^\times a$. Thus $[a]$ records only the direction of $a$, identifying all its nonzero scalar multiples. For $V_0=F^t$, a set $D\subseteq PG(t-1,F)$ is \emph{spanning} if its representatives span $V_0$. Write
\[
 \Delta(D)
 =
 \min_{0\ne c\in V_0^\vee}
 |\{[a]\in D:c(a)\ne0\}|.
\]
For $a\in V_0\setminus\{0\}$, say that $f:V_0\to F$ is \emph{affine in direction $[a]$} if, for every $x\in V_0$, there exist $A,B\in F$ such that
\[
 f(x+\alpha a)=A+B\alpha
 \qquad\text{for every }\alpha\in F.
\]
Replacing $a$ by $ca$ for any $c\in F^\times$ does not change this condition, so it depends only on $[a]$.  Define $\cA(D)$ to be the set of functions affine in every direction in $D$, which is an $F$-linear subspace of $F^{V_0}$.

\begin{example}\label{ex:two-dimensional-direction-affine}
Take $V_0=F^2$, with coordinates $(u,v)$, and assume $|F|>2$. For $D=\{[e_1],[e_2]\}$, we have $\Delta(D)=1$ and
\[
 \cA(D)
 =\{(u,v)\mapsto c_0+c_1u+c_2v+c_{12}uv:
     c_0,c_1,c_2,c_{12}\in F\}.
\]
Thus the two axis directions allow the mixed term $uv$. If $D'=D\cup\{[e_1+e_2]\}$, then $\Delta(D')=2$, while on a diagonal line
\[
 (u+\alpha)(v+\alpha)
 =uv+\alpha(u+v)+\alpha^2,
\]
so the diagonal direction excludes $uv$, and $\cA(D')=\{(u,v)\mapsto c_0+c_1u+c_2v\}$.

Over $\mathbb R$, the analogous condition is $(a\mathbin{\cdot}\nabla)^2f=0$. The axis conditions permit $uv$, whereas the diagonal condition excludes it since $(\partial_u+\partial_v)^2(uv)=2$. In characteristic two, restrictions to lines replace this derivative test.
\end{example}

The next lemma gives explicit sets of $O_q(t)$ directions with a lower bound on $\Delta(D)/|D|$ and exponentially many linearly independent functions in $\cA(D)$.

\begin{lemma}
\label{lem:direction-affine-family}
Let $q=2^h\ge8$ be fixed and let $F=\F_{q^2}$. For arbitrarily large integers $t$, one can explicitly construct a spanning set $D\subseteq PG(t-1,F)$ such that, writing $\nu=|D|$ and $\Delta=\Delta(D)$,
\begin{equation}
 \nu=O_q(t),
 \qquad
 \frac{\Delta}{\nu}
 \ge \frac{q-5}{q-2},
 \qquad
 \dim_F\cA(D)\ge 2^{t/3}.
 \label{eq:uniform-direction-affine-input}
\end{equation}
\end{lemma}
The proof is postponed to Subsection~\ref{proof in 3}.

\subsection{The relation matrix construction}
\label{subsec:cone-line-relation-matrix}

Fix $F=\F_{q^2}$ as above, and set $V_0=F^t$.  Consider the triples $\mathcal L_F=\bigl\{\{p,r,p+r\}:p,r\in F^\times,\ p\ne r\bigr\}$.

Let $D\subseteq PG(t-1,F)$ be spanning.  Write $\nu=|D|$, $\Delta=\Delta(D)$, and $\widehat D=\{b\in V_0\setminus\{0\}:[b]\in D\}$. Thus $\widehat D$ is the set of all nonzero representatives of the points of $D$, with $L:=|\widehat D|=(q^2-1)\nu$ and $M:=(q^2-1)q^{2t}$. Index the columns by $\Omega=F^\times\times V_0$.  For every $b\in\widehat D$, $x\in V_0$, and $\ell\in\mathcal L_F$, include the row whose support is $T(b,x,\ell)=\{(p,x+pb):p\in\ell\}$.  Let $H=H_D$ be the resulting binary matrix. The set $\widehat D$ is invariant under multiplication by $F^\times$; we use this invariance to compute the spectrum below.

\begin{example}
\label{ex:toy-compiler}
Take the smallest allowed parameter $q=8$, let $F=\F_{64}$, choose $\omega\in F\setminus\F_2$, take $t=1$, and set $D=\{[1]\}\subseteq PG(0,F)$.  Then $\widehat D=F^\times$ and $\Omega=F^\times\times F$.  Choose $\ell=\{1,\omega,1+\omega\}$, the direction $b=1$, and the base point $x=1$.  The matrix $H_D$ contains the row
\[
 T(1,1,\ell)
 =\{(1,0),(\omega,1+\omega),(1+\omega,\omega)\}.
\]
This incidence row imposes a relation among the images of the corresponding standard basis vectors in the quotient.  If $s_{(p,y)}$ denotes the class of the corresponding coordinate vector in $\Gamma=\F_2^\Omega/\operatorname{rowspan}_{\F_2}H$, then
\[
 s_{(1,0)}+s_{(\omega,1+\omega)}
 +s_{(1+\omega,\omega)}=0.
\]
Once Lemma~\ref{lem:compiler-kernel} shows that these generators are nonzero and distinct, this row gives the triangle $\{g,g+s_{(1,0)},g+s_{(\omega,1+\omega)}\}$ for every $g\in\Gamma$.

\end{example}

\begin{remark}\label{rem:labels-versus-generators}
The three column labels in $T(b,x,\ell)$ sum to $(0,x)$ in $F\times V_0$, not necessarily zero. It is the corresponding generators in $\Gamma$ that sum to zero, by the relation imposed by this row of $H$.
\end{remark}

\paragraph{Row and column weights.} We first show that each pair of columns occurs together in at most one row, then count the rows containing each column. It follows that $\widetilde H_D^{\mathsf T}\widetilde H_D-\rho I$ is the adjacency matrix of a simple regular graph.

\begin{lemma}
\label{lem:compiler-combinatorics}
The matrix $H_D$ defined above belongs to $\F_2^{R\times M}$, where $R=(q^2-2)LM/6$.  Its rows are distinct and have weight three, every column has weight $\rho=(q^2-2)L/2$, and no pair of columns occurs together in more than one row.  Consequently, $\widetilde H_D^{\mathsf T}\widetilde H_D-\rho I$ is the adjacency matrix of a simple $2\rho$-regular graph, and $P_D:=P_{H_D}$, as defined in \eqref{eq:matrix-link-operator}, is its normalized adjacency operator.
\end{lemma}

\begin{proof}
Each relation is a triple.  If a row contains $(p,y)$ and $(r,z)$, where $p\ne r$, then $b=(p+r)^{-1}(y+z)$.  The first coordinates $p$ and $r$ determine $\ell$, and then $x=y+pb$.  Hence a pair of coordinates determines its row; in particular, no pair occurs in two rows and all rows are distinct.  For a fixed coordinate $(p,y)$, there are $L$ choices of $b$ and $(q^2-2)/2$ triples $\ell\in\mathcal L_F$ containing $p$, since the $q^2-2$ choices of $r\ne0,p$ count each such triple twice. Thus its column weight is $\rho=(q^2-2)L/2$, and double counting incidences gives $R=\rho M/3=(q^2-2)LM/6$.

The preceding uniqueness makes every off-diagonal entry of $\widetilde H_D^{\mathsf T}\widetilde H_D$ either zero or one.  Each of the $\rho$ triples through a coordinate contributes two distinct neighbors, so the resulting graph is simple and $2\rho$-regular. Dividing its adjacency matrix by $2\rho=(q^2-2)L$ gives $P_D$.
\end{proof}

\paragraph{Kernel and minimum distance.} We next obtain binary kernel vectors by evaluating functions in $\cA(D)$ on $\Omega$ and applying the field trace.

\begin{lemma}
\label{lem:compiler-kernel}
The map
\[
 \Phi:\cA(D)\longrightarrow\F_2^\Omega,
 \qquad
 \Phi(u)_{(p,x)}
 =\Tr\bigl(pu(x)\bigr)
\]
is an injective $\F_2$-linear map into $\ker H_D$, so $\dim_{\F_2}\ker H_D\ge 2h\dim_F\cA(D)$. The binary row space of $H_D$ has minimum distance at least three.
\end{lemma}

\begin{proof}
For $u\in\cA(D)$, write its restriction in direction $b$ as $u(x+\alpha b)=A+B\alpha$. Every $\ell\in\mathcal L_F$ satisfies $\sum_{s\in\ell}s=\sum_{s\in\ell}s^2=0$.  Hence the row parity is
\[
 \Tr\left(
 A\sum_{s\in\ell}s+B\sum_{s\in\ell}s^2\right)=0.
\]
Thus $\Phi(u)\in\ker H$.  If $\Phi(u)=0$, then for every $x\in V_0$ the trace pairing annihilates $pu(x)$ for all $p\in F^\times$; nondegeneracy gives $u(x)=0$.  Therefore $\Phi$ is injective and $\dim_{\F_2}\ker H\ge 2h\dim_F\cA(D)$.

For any $(p,x)$, a suitable constant function has nonzero image at that coordinate: by nondegeneracy of the trace pairing, choose $C\in F$ with $\Tr(pC)=1$.  If $p\ne r$, choose instead $C$ with $\Tr((p+r)C)=1$; the constant function $C$ separates $(p,x)$ and $(r,y)$.  If the first coordinate is $p$ in both cases but $x\ne y$, choose an $F$-linear functional $a:V_0\to F$ such that $\Tr(pa(x+y))=1$. The globally linear function $a$ lies in $\cA(D)$, so its image separates $(p,x)$ and $(p,y)$.  A vector of weight one in the row space would force every kernel word to vanish at one coordinate, while a vector of weight two would force every kernel word to agree at two coordinates.  Both conclusions contradict the separation properties above.  Hence the row space has minimum distance at least three, since $\ker H_D=(\operatorname{rowspan}_{\F_2}H_D)^\perp$.
\end{proof}

\paragraph{Spectrum.} The graph with adjacency matrix $\widetilde H_D^{\mathsf T}\widetilde H_D-\rho I$ is a categorical product of a complete graph and a Cayley graph. We compute the spectrum from the eigenvalues of these two factors.

\begin{lemma}
\label{lem:compiler-spectrum}
The graph defined in Lemma~\ref{lem:compiler-combinatorics} is connected, and its normalized adjacency operator $P_D$ satisfies
\begin{equation}
\label{eq:compiler-spectrum-general}
 \Spec(P_D|_{\one^\perp})
 \subseteq
 \left[
  -\frac1{q^2-2},
  \max\left\{
    1-\frac{q^2\Delta}{L},
    \frac1{(q^2-2)(q^2-1)}
  \right\}
 \right].
\end{equation}
\end{lemma}

\begin{proof}
By Lemma~\ref{lem:compiler-combinatorics}, two columns $(p,x)$ and $(r,y)$ are adjacent if and only if $p\ne r$ and $(p+r)^{-1}(x+y)\in\widehat D$. Since $\widehat D$ is closed under multiplication by $F^\times$, this is equivalent to $p\ne r$ and $x+y\in\widehat D$.
Thus the graph is the categorical product
\[
 K_{q^2-1}\times\operatorname{Cay}(V_0,\widehat D),
 \qquad
 P_D=P_{K_{q^2-1}}\otimes P_{\operatorname{Cay}(V_0,\widehat D)}.
\]

For $c\in V_0^\vee$, define $\chi_c(x)=(-1)^{\Tr(c(x))}$. The nondegenerate trace pairing identifies $V_0^\vee$ with the additive characters of $V_0$. By Lemma~\ref{lem:cayley-character-spectrum}, the eigenvalue of $P_{\operatorname{Cay}(V_0,\widehat D)}$ at $\chi_c$ is
\[
 \mu(c)
 =\frac1L\sum_{b\in\widehat D}\chi_c(b)
 =1-\frac{q^2w(c)}L,
\]
where $w(c)$ is the number of $[a]\in D$ with $c(a)\ne0$. Indeed, the representatives of a direction $[a]$ with $c(a)=0$ contribute $q^2-1$, while those of a direction with $c(a)\ne0$ contribute $-1$.

The normalized adjacency eigenvalues of $K_{q^2-1}$ are $1$ and $-1/(q^2-2)$, the latter with multiplicity $q^2-2$.  Taking products gives the exact eigenvalues
\begin{equation}
 \lambda_0(c)=\mu(c)=1-\frac{q^2w(c)}L,
 \qquad
 \lambda_1(c)=-\frac{\mu(c)}{q^2-2}
 =\frac{q^2w(c)-L}{(q^2-2)L}.
 \label{eq:exact-compiler-eigenvalues}
\end{equation}
In the invariant subspace corresponding to $c$, the eigenvalues are $\lambda_0(c)$ once and $\lambda_1(c)$ repeated $q^2-2$ times, with multiplicities added when they coincide.

For $c=0$, the two eigenvalues in \eqref{eq:exact-compiler-eigenvalues} are $1$ and $-1/(q^2-2)$; only the first has a constant eigenvector.  For $c\ne0$, the definition of $\Delta(D)$ gives $\Delta\le w(c)\le|D|$.  Therefore
\[
 \lambda_0(c)\le1-\frac{q^2\Delta}{L},
 \qquad
 -\frac1{q^2-2}\le\lambda_1(c)
 \le\frac1{(q^2-2)(q^2-1)},
\]
while $\lambda_0(c)\ge-1/(q^2-1)$.  These bounds give \eqref{eq:compiler-spectrum-general}.  Since $D$ spans $V_0$, one has $\Delta\ge1$, so the upper endpoint in \eqref{eq:compiler-spectrum-general} is strictly below one.  For a finite regular graph, the multiplicity of the normalized-adjacency eigenvalue one is the number of connected components.  The graph is therefore connected.
\end{proof}

\subsection{Proof of Lemma~\ref{lem:matrix-family}}\label{proof in 2}

\begin{proof}
Fix $\lambda>0$, choose $q=2^h\ge8$ so that $4/(q-2)<\lambda$, set $F=\F_{q^2}$, and take any set $D$ supplied by Lemma~\ref{lem:direction-affine-family}. Construct $H=H_D$ as above, so $M=(q^2-1)q^{2t}$. Lemma~\ref{lem:compiler-combinatorics} gives $\rho=(q^2-2)(q^2-1)|D|/2=O_\lambda(t)$, $R=\rho M/3$, and distinct weight-three rows. Lemma~\ref{lem:compiler-kernel} gives row-space distance at least three and
\[
 \dim_{\F_2}\ker H
 \ge 2h\dim_F\cA(D)
 \ge 2h\,2^{t/3}.
\]

Since $L=(q^2-1)\nu$, the lower bound on $\Delta/\nu$ in \eqref{eq:uniform-direction-affine-input} gives
\[
 1-\frac{q^2\Delta}{L}
 \le 1-\frac{\Delta}{\nu}
 \le \frac3{q-2}
 <\frac4{q-2}.
\]
Together with $1/(q^2-2)<4/(q-2)<\lambda$, Lemma~\ref{lem:compiler-spectrum} and \eqref{eq:compiler-spectrum-general} yield $\|P_D|_{\one^\perp}\|\le\lambda$.  Simplicity and regularity follow from Lemma~\ref{lem:compiler-combinatorics}, and connectivity follows from Lemma~\ref{lem:compiler-spectrum}. The set $D$ is explicit, and, since $F$ is fixed, all scalar representatives and rows of $H$ can be enumerated by finite-field arithmetic in time polynomial in the output size.
\end{proof}

%% file: 04-ag-direction-system.tex
\section{Direction sets from algebraic curves}
\label{sec:ag-directions}

We prove Lemma~\ref{lem:direction-affine-family} using a Garcia--Stichtenoth tower. Evaluation at rational points gives the directions. Multiplication of functions in Riemann--Roch spaces gives a matrix of linear forms whose maximal minors belong to $\cA(D)$. This section uses several facts from algebraic geometry. Readers unfamiliar with them may consult~\cite[Chapter~1]{Sti09}, or Appendix~\ref{app:ag-background} for a brief review.

Fix $q=2^h\ge8$ and $F=\F_{q^2}$. For a set $D\subseteq PG(t-1,F)$, we use $\Delta(D)$ and $\cA(D)$ as defined in Subsection~\ref{subsec:direction-affine-input}.

We recall the basic notation for curves. Let $X/F$ be a smooth projective geometrically integral curve, with function field $F(X)$ and genus $g$. Informally, smoothness excludes singularities, projectivity includes the points at infinity, and geometric integrality means that the curve remains irreducible and reduced over an algebraic closure. An $F$-rational point $P\in X(F)$ allows evaluation $f(P)\in F$ whenever $f\in F(X)$ has no pole at $P$. A divisor $C$ is a finite integer combination of closed points; its degree is the sum of the coefficients weighted by the degrees of the points over $F$. In particular, rational points have degree one. For $f\ne0$, the divisor $\operatorname{div}(f)$ records the orders of its zeros and poles, with poles having negative coefficients.

The Riemann--Roch space is $L(C)=\{0\}\cup\{f\in F(X)^\times:\operatorname{div}(f)+C\ge0\}$, and we write $\ell(C)=\dim_F L(C)$. Thus $L(mP_\infty)$ consists of functions with no poles except possibly a pole of order at most $m$ at $P_\infty$. The genus enters through Riemann--Roch: $\ell(C)\ge\deg C+1-g$, with equality when $\deg C>2g-2$. We also use $L(A)L(B)\subseteq L(A+B)$ and the fact that a nonzero function in $L(mP_\infty)$ has at most $m$ rational zeros away from $P_\infty$. Appendix~\ref{app:ag-background} gives further explanation and references for these facts.

\subsection{Directions from the Garcia--Stichtenoth tower}
\label{subsec:ag-evaluation-directions}

\begin{lemma}[Garcia--Stichtenoth tower]\cite{GS96}
\label{lem:gs-tower-input}
There is an explicit sequence of smooth projective geometrically integral curves $X_r/F$ of genera $g_r$, together with points $P_{\infty,r}\in X_r(F)$ and explicitly enumerable sets $\mathcal P_r\subseteq X_r(F)\setminus\{P_{\infty,r}\}$, such that
\[
 g_r\longrightarrow\infty,
 \qquad
 \frac{|\mathcal P_r|}{g_r}\longrightarrow q-1.
\]
\end{lemma}

For fixed $q$, the points $P_{\infty,r}$, the sets $\mathcal P_r$, and bases of the one-point Riemann--Roch spaces used below can be computed in time polynomial in the input and output sizes~\cite{GX22}; the basis algorithm used there is due to~\cite{Shum01}.

Every sufficiently large level has $|\mathcal P_r|\ge(q-2)g_r$. Fix such a level with $g_r\ge4$, abbreviate its curve and genus by $X$ and $g$, let $P_\infty=P_{\infty,r}$, and choose distinct $P_1,\ldots,P_\nu\in\mathcal P_r$, where $\nu=(q-2)g$. Set $G=3gP_\infty$, $W=L(G)$, and $t=\dim_F W=2g+1$. The equality $t=2g+1$ follows from Riemann--Roch \cite{Sti09}.  Fix an $F$-linear identification $W^\vee\cong F^t$.  For $P\in X(F)\setminus\{P_\infty\}$, let $e_P\in W^\vee$ be the evaluation functional, $e_P(s)=s(P)$, and define $D=\{[e_{P_1}],\ldots,[e_{P_\nu}]\}\subseteq PG(t-1,F)$.

Since $\nu>\deg G=3g$, the evaluation map $W\to F^\nu$ is injective \cite[Corollary~2.2.3]{Sti09}, so the $e_{P_j}$ span $W^\vee$.  Because $1\in W$, proportional evaluation functionals must be equal.  For distinct selected points $P\ne R$, Riemann--Roch gives $\ell(G-R)=2g$ and $\ell(G-P-R)=2g-1$ \cite[Theorem~1.5.17]{Sti09}.  Hence some $s\in L(G-R)\setminus L(G-P-R)$ satisfies $s(R)=0$ and $s(P)\ne0$, so $e_P\ne e_R$.  Thus the projective points $[e_P]$ are distinct and $|D|=\nu=(q-2)g=O_q(t)$.  Finally, every hyperplane in $W^\vee$ has the form $\{x:x(s)=0\}$ for some $0\ne s\in W$. Such an $s$ has at most $\deg G=3g$ zeros among the selected points, so $\Delta(D)\ge \nu-3g=(q-5)g$.

\begin{remark}
\label{rem:square-field}
A reader may wonder why we work over the square field $F=\F_{q^2}$. The Drinfeld--Vl\u{a}du\c{t} bound gives an asymptotic ratio of rational points to genus at most $q-1$~\cite[Theorem~7.1.3]{Sti09}, and the towers of Garcia and Stichtenoth~\cite{GS95,GS96} attain it. This leaves enough evaluation points to obtain the relative distance $(q-5)/(q-2)$.
\end{remark}

\begin{example}
Let $z$ be the affine coordinate on $\mathbb P^1_F$, let $P_\infty$ be the point at infinity, and take
\[
 W_2=L(2P_\infty)
 =\operatorname{span}_F\{1,z,z^2\}.
\]
For four distinct $a_1,\ldots,a_4\in F$, evaluation gives $e_{a_i}=(1,a_i,a_i^2)\in W_2^\vee$, and any three of these vectors span by the Vandermonde determinant. Since a nonzero quadratic has at most two roots, $D=\{[e_{a_i}]:i\in[4]\}$ has $\Delta(D)=2$, with equality witnessed by $(z-a_1)(z-a_2)$. The tower replaces this finite example by a family with $t\to\infty$.
\end{example}

\subsection{Maximal minors and affine restrictions}
\label{sec:determinantal-affine-functions}

Set $k=\lfloor g/4\rfloor$ and $m=3k$, and define $A=(g+k)P_\infty$ and $B=G-A$. The Riemann--Roch lower bound $\ell(C)\ge\deg C+1-g$~\cite[Theorem~1.5.15]{Sti09} gives $\ell(A)\ge k+1$ and $\ell(B)\ge g-k+1\ge m+1$. Choose $U\le L(A)$ with $\dim_FU=k$ and $V\le L(B)$ with $\dim_FV=m$. Since $A+B=G$, multiplication defines the $F$-bilinear map $U\times V\longrightarrow W$, $(u,z)\longmapsto uz$.

Here $W=L(G)$ consists of functions on the curve, whereas $\cA(D)$ consists of functions on $W^\vee$. We associate a matrix of linear forms with the multiplication map $U\times V\to W$. Choose bases $u_1,\ldots,u_k$ of $U$ and $z_1,\ldots,z_m$ of $V$. For $x\in W^\vee$, define
\[
 \mathcal M(x)=\bigl(x(u_i z_j)\bigr)_{1\le i\le k,\,1\le j\le m}.
\]
For $J\subseteq[m]$ with $|J|=k$, let $p_J(x)=\det \mathcal M_J(x)$, where $\mathcal M_J$ is the square submatrix with column set $J$.

\begin{example}
The rank-one calculation is illustrated on $\mathbb P^1_F$.  Take
\[
 W_4=L(4P_\infty)=\operatorname{span}_F\{1,z,z^2,z^3,z^4\},
 \quad
 U_1=L(P_\infty),
 \quad
 V_3=L(3P_\infty),
\]
so that $U_1V_3\subseteq W_4$.  Using the monomial bases and writing $x_j=x(z^j)$, the resulting matrix of linear forms is
\[
 \mathcal M(x)
 =
 \begin{pmatrix}
  x_0&x_1&x_2&x_3\\
  x_1&x_2&x_3&x_4
 \end{pmatrix}.
\]
At the affine point $P_a$, evaluation gives the rank-one matrix
\[
 \mathcal M(e_a)
 =
 \begin{pmatrix}1\\a\end{pmatrix}
 \begin{pmatrix}1&a&a^2&a^3\end{pmatrix},
\]
so adding $\alpha e_a$ to $x$ adds a rank-one matrix.  For example, the minor in the first two columns is $p_{12}(x)=x_0x_2+x_1^2$, and in characteristic two
\[
 \begin{aligned}
 p_{12}(x+\alpha e_a)
 &=(x_0+\alpha)(x_2+\alpha a^2)
   +(x_1+\alpha a)^2\\
 &=p_{12}(x)+\alpha(x_2+a^2x_0).
 \end{aligned}
\]
The two copies of $\alpha^2a^2$ cancel.  Thus a quadratic function on $W_4^\vee$ becomes affine on every line $x+F e_a$.
\end{example}

\begin{lemma}
\label{lem:rank-one-restriction}
For every selected point $P$, $x\in W^\vee$, and $\alpha\in F$,
\begin{equation}
 \mathcal M(x+\alpha e_P)
 =\mathcal M(x)+\alpha\,\mathbf u(P)\mathbf z(P)^{\mathsf T},
 \label{eq:rank-one-update}
\end{equation}
where $\mathbf u(P)=(u_i(P))_i$ and $\mathbf z(P)=(z_j(P))_j$.  Consequently every $p_J$ belongs to $\cA(D)$.
\end{lemma}

\begin{proof}
Each matrix entry changes by $\alpha u_i(P)z_j(P)$, which proves \eqref{eq:rank-one-update}.  The rank-one determinant identity $\det(C+\alpha ab^{\mathsf T}) =\det C+\alpha b^{\mathsf T}\operatorname{adj}(C)a$ shows that every maximal minor is affine in $\alpha$.
\end{proof}

\begin{proposition}
\label{prop:maximal-minor-independence}
The functions
\[
 \{p_J:W^\vee\to F:J\subseteq[m],\ |J|=k\}
\]
are $F$-linearly independent.
\end{proposition}

\begin{proof}
\emph{Polynomial independence.} Extend scalars to an algebraic closure $\overline F$, and write $W_{\overline F}=W\otimes_F\overline F$ and $X_{\overline F}=X\times_F\overline F$. Regard $\mathcal M$ as a matrix of linear forms in the polynomial ring $S=\operatorname{Sym}_{\overline F}(W_{\overline F})$. For nonzero $\alpha\in\overline F^k$ and $\beta\in\overline F^m$, its generalized entry is $(\sum_i\alpha_i u_i)(\sum_j\beta_j z_j)$. Both factors are nonzero, since the chosen bases remain linearly independent after extending scalars, and their product is nonzero because $X_{\overline F}$ is integral. Thus $\mathcal M$ is $1$-generic.

Let $I\subseteq S$ be the ideal of maximal minors and let $\mathfrak m=S_{>0}$ be the homogeneous maximal ideal. Theorem~6.4 of~\cite{Eis05} gives $\operatorname{ht}I=m-k+1$. The local ring $S_{\mathfrak m}$ is regular, so $\operatorname{grade}(IS_{\mathfrak m})=\operatorname{ht}(IS_{\mathfrak m})=m-k+1$. All matrix entries lie in its maximal ideal; hence Corollary~A2.61 of~\cite{Eis05} shows that the $\binom{m}{k}$ maximal minors form a minimal generating set of $IS_{\mathfrak m}$. A nonzero scalar linear relation would express one minor in terms of the others, contradicting minimality. The minors are therefore linearly independent as polynomials over $\overline F$, and hence over $F$.

\emph{Independence as functions.} Over a finite field, polynomial independence alone does not imply independence of the induced functions. Choose selected evaluation functionals $e_{R_1},\ldots,e_{R_t}$ that form a basis of $W^\vee$, using the spanning property proved above.  Write $x=\sum_{\ell=1}^t c_\ell e_{R_\ell}$.  Fix $\ell$ and regard the other $c_j$ as indeterminates.  Then
\[
 \mathcal M(x)=C+c_\ell\,\mathbf u(R_\ell)\mathbf z(R_\ell)^{\mathsf T}
 \qquad
 \text{over }F[c_j:j\ne\ell].
\]
The rank-one determinant identity shows that every maximal minor has degree at most one in $c_\ell$.  Since this holds for every $\ell$, every linear combination of the $p_J$ is multilinear in $c_1,\ldots,c_t$.  A multilinear polynomial that vanishes on $\{0,1\}^t$ is zero, by induction on $t$.  Hence a combination that vanishes as a function on $W^\vee$ is the zero polynomial, and polynomial independence forces all its coefficients to vanish.
\end{proof}

\subsection{Proof of Lemma~\ref{lem:direction-affine-family}}\label{proof in 3}

\begin{proof}
Use the direction set constructed in Subsection~\ref{subsec:ag-evaluation-directions}.

Lemma~\ref{lem:rank-one-restriction} shows that all maximal minors belong to $\cA(D)$, while Proposition~\ref{prop:maximal-minor-independence} shows that they are linearly independent as functions. Consequently,
\[
 \dim_F\cA(D)
 \ge\binom{3k}{k}
 \ge\frac{(27/4)^k}{3k+1}
 \ge 2^{t/3}
\]
for all sufficiently large $g$. The middle inequality follows because $\binom{3k}{k}2^{2k}$ is the largest term in the expansion of $(1+2)^{3k}$. For the last inequality, $t=2g+1=8k+O(1)$ and $\log_2(27/4)>8/3$. Discarding finitely many earlier tower levels therefore gives the claimed infinite family.

The tower algorithms compute the selected rational points and a basis of $W$, and evaluation then produces $D$ by finite-field linear algebra.
\end{proof}

%% file: 05-weighted-dimension-lift.tex
\section{Linear-degree Cayley complexes from graph products}
\label{sec:dimension-lift}

In this section, we derive Theorem~\ref{thm:main-higher-dimensional} from the graph-product construction in~\cite{Gol21}, which refines~\cite{LMY20}. The argument is independent of Sections~\ref{sec:relation-matrices} and~\ref{sec:ag-directions}.

\begin{lemma}
\label{lem:universal-negative-eigenvalue}
Let $\ell\ge1$, let $(X,\mu_\ell)$ be a pure weighted $\ell$-dimensional simplicial complex, and let $P$ be the random-walk operator of its weighted one-skeleton.  Then every eigenvalue of $P$ is at least $-1/\ell$.
\end{lemma}

\begin{proof}
The induced measures satisfy $\pi=\mu_0$ and $P(u,v)=\mu_1(\{u,v\})/(2\mu_0(u))$ on edges. For every real function $f$ on the vertices,
\[
 \langle f,Pf\rangle_\pi+\frac1\ell\lVert f\rVert_\pi^2
 =\frac1{\ell(\ell+1)}
   \mathbb E_{\tau\sim\mu_\ell}
   \left(\sum_{v\in\tau}f(v)\right)^2\ge0.
\]
The Rayleigh-quotient characterization gives the result.
\end{proof}

\paragraph{The product complex.} Let $G=(V,E,w)$ be a finite connected loopless graph with positive edge weights, let $K\ge2$, and let $s\ge2K$.  Define $Z=Z(G;K,s)$ on $V\times[s]$ as the downward closure of
\[
 Z(K)=\left\{\{(v_i,b_i):0\le i\le K\}:
 \begin{array}{l}
  \{v_0,\ldots,v_K\}=\{u,v\}\text{ for some }\{u,v\}\in E,\\
  b_0,\ldots,b_K\text{ are pairwise distinct}
 \end{array}\right\}.
\]
For $\sigma\in Z(K)$, let $j=|\sigma\cap(\{u\}\times[s])|$, where $\{u,v\}$ is its projected edge, and define
\[
 m_K(\sigma)=\frac{w(\{u,v\})}{\binom{K-1}{j-1}}.
\]
This is independent of the ordering of $u,v$. Normalize $m_K$ to a top-face probability distribution $\mu_K$, and use the induced measures on lower-dimensional faces and links. Under the balanced-weight convention of~\cite{Gol21}, the weights $m_i$ are defined recursively: the weight of each face below the top dimension is the sum of the weights of the faces of dimension one higher that contain it. Iterating this rule gives, for $\sigma\in Z(i)$,
\[
 \begin{aligned}
 m_i(\sigma)
 &=(K-i)!\sum_{\substack{\tau\in Z(K)\\\sigma\subseteq\tau}}m_K(\tau)\\
 &=(K-i)!\binom{K+1}{i+1}
   \left(\sum_{\tau\in Z(K)}m_K(\tau)\right)\mu_i(\sigma).
 \end{aligned}
\]
Thus they differ from the induced measures only by a constant at each dimension and define the same link walks.

\begin{theorem}\cite[Theorem~18]{Gol21}
\label{thm:golowich-product-links}
Let $Z=Z(G;K,s)$ be the product complex above, where $K\ge2$, $s\ge2K$, and $|V|\ge4$. For every $0\le t\le K-2$ and every $F\in Z(t)$, the second-largest eigenvalue of the random-walk operator of the weighted one-skeleton of $Z_F$ is at most $1/(t+2)$.  In particular, every such link graph is connected.
\end{theorem}

\begin{proof}[Proof of Theorem~\ref{thm:main-higher-dimensional}]
Fix $d\ge2$, set $K=2d$, and choose a power of two $s=2^a\ge2K$, where $a$ depends only on $d$. For each $k\ge2$, let $G=\Cay(\F_2^k,\{e_1,\ldots,e_k\})$ be the $k$-dimensional cube with uniform edge weights, identify $[s]$ with the additive group $A=\F_2^a$, and let $Z=Z(G;K,s)$ with top-face distribution $\mu_K$.

Let $Y$ be the $d$-skeleton of $Z$, equipped with the induced top-face distribution
\[
 \mu_d(T)=\binom{K+1}{d+1}^{-1}
 \sum_{\substack{\sigma\in Z(K)\\T\subseteq\sigma}}\mu_K(\sigma),
 \qquad T\in Z(d).
\]
Equivalently, sample a $K$-face according to $\mu_K$, then choose a uniform $(d+1)$-subset. The complex $Y$ is simple and pure, and $\mu_d$ has full support. Translation in $\F_2^k\times A$ preserves cube edges, distinctness of labels, and all face weights, so the weighted complex is translation invariant. Its one-skeleton is the Cayley graph on $\F_2^k\times A$ with generator set
\[
 S=\bigl(\{0,e_1,\ldots,e_k\}\bigr)
       \times\bigl(A\setminus\{0\}\bigr).
\]
Indeed, two vertices in the same fiber or in adjacent cube fibers form an edge exactly when their labels are distinct.  The set $S$ spans: for any $b\ne0$, one has $(e_i,b)+(0,b)=(e_i,0)$. Hence the one-skeleton is connected. Writing $n=k+a$, its degree is
\[
 |S|=(s-1)(k+1)=(s-1)(n-a+1)=\Theta_d(n).
\]
This order is optimal among connected Cayley complexes over $\F_2^n$ with connected vertex links, whose degree is at least $2n-1$~\cite[Corollary~181]{DLW25}.

It remains to check the links. Fix $0\le t\le d-2$, let $F\in Y(t)$, and, for a link edge $\{u,v\}$, define its unnormalized weight by $w_F^Y(u,v)=\sum_{T\supseteq F\cup\{u,v\}}\mu_d(T)$.  Counting the $(d+1)$-subsets of each $K$-face that contain $F\cup\{u,v\}$ gives
\[
 w_F^Y(u,v)
 =
 \frac{\binom{K-t-2}{d-t-2}}{\binom{K+1}{d+1}}
 \sum_{\substack{\sigma\in Z(K)\\F\cup\{u,v\}\subseteq\sigma}}
 \mu_K(\sigma).
\]
The binomial factor is independent of $u,v$. After normalization, the link walk of $Y_F$ is therefore identical to the random walk on the weighted one-skeleton of $Z_F$.

Now take $t=d-2$.  Theorem~\ref{thm:golowich-product-links} bounds the second-largest eigenvalue of this walk by $1/d$.  The link $Z_F$ has dimension $K-|F|=2d-(d-1)=d+1$, so Lemma~\ref{lem:universal-negative-eigenvalue} bounds its least eigenvalue by $-1/(d+1)$.  Consequently every nontrivial eigenvalue lies in $[-1/(d+1),1/d]$, which proves the required two-sided norm bound.  For every $0\le t\le d-2$, Theorem~\ref{thm:golowich-product-links} gives a second eigenvalue at most $1/(t+2)<1$, so each positive-dimensional link is connected.  Together with the connected one-skeleton, this proves all connectivity assertions.

The construction is explicit: since $K,s,a$ depend only on $d$, its $O_d(k)$ orbit representatives and exact rational weights can be enumerated from the $k$ cube directions in time polynomial in $n$. Letting $k\to\infty$ gives the required family.
\end{proof}

The induced weighted one-skeleton walk of $Y$ is also the same as that of $Z$. Since the cube walk has spectral gap $2/k$, the global part of~\cite[Theorem~18]{Gol21} gives the gap $2/(k\sum_{j=1}^{2d}1/j)$, which tends to zero. Thus these one-skeleta are connected but do not form an expander family.

%% file: 06-open-problems-and-limitations.tex
\section{Discussion and open problems}
\label{sec:open-problems}
\begin{itemize}
  \item \textbf{The range $0<\lambda<1/d$.} For fixed $d\ge3$, can the codimension-two bound $1/d$ in Theorem~\ref{thm:main-higher-dimensional} be improved while retaining polynomial Cayley degree? For the cube graphs used in Section~\ref{sec:dimension-lift}, the $d$-skeleton has a codimension-two link with eigenvalue $1/d$, even if we increase the dimension of the product complex before taking the skeleton. Thus increasing that dimension cannot give a local spectral norm below $1/d$.
  \item \textbf{Near-linear or linear Cayley degree.} Theorem~\ref{thm:main-higher-dimensional} gives optimal-order degree at the endpoint $1/d$. For fixed $0<\lambda<1/2$, can the two-dimensional degree in Theorem~\ref{thm:main-two-dimensional} be reduced to $n^{1+o(1)}$, or $O_\lambda(n)$? The exponent $6h$ in \eqref{eq:main-degree} comes from $\Phi(\cA(D))$. Improving it may require more functions in $\cA(D)$, kernel vectors outside this subspace, or fewer columns for the same nullity.
  \item \textbf{Stronger expansion notions.} We do not bound cofilling constants or establish cosystolic expansion. Can polynomial-degree abelian Cayley complexes satisfy the stronger assumptions used in low-soundness agreement testing, including coboundary or cosystolic expansion of associated faces complexes~\cite{DDL24,BLM24,HR26}?
\end{itemize}

%% file: Aknowledgement.tex
\section*{Statement on the use of AI}

The author used ChatGPT 5.6 Family (mainly Pro and Sol) for assistance with the proofs; Appendix~\ref{app:ai-research-note} describes the details. A Lean formalization written by Codex is available in~\cite{Mao26} for references. The author independently checked, simplified, and organized every proof and takes full responsibility for all statements and results in the paper.

\section*{Acknowledgments}
The author is grateful to Xin Li for the introduction of high-dimensional expanders and for early discussions. The author would like to thank Zihong Lin, Michael Ruofan Zeng, and a differential geometer who wishes to remain anonymous for helpful discussions and generous assistance.

%% file: appendix-ai-contributions.tex
\section{AI contributions and research note}
\label{app:ai-research-note}

We had worked on this problem for a long time. Although numerical experiments supported the feasibility of the parameters in Lemma~\ref{lem:matrix-family}, we had not found a construction achieving them. The ideas and proof techniques in Sections~\ref{sec:relation-matrices} and~\ref{sec:ag-directions} were subsequently developed by ChatGPT~5.6.

We gave ChatGPT approximate versions of the bounds in Lemma~\ref{lem:matrix-family} and asked for a matrix satisfying them, initially in dimension two. A finite example with unexpectedly large binary nullity led ChatGPT to identify multiplication over $\F_8$ in the construction. The additional kernel vectors came from functions affine along certain finite-field directions. Further experiments gave small direction sets with both a distance bound and additional kernel vectors, but did not establish the asymptotic bounds needed for an infinite family.

Coordinate directions gave exponentially large kernels but failed the spectral bound; other tested directions met this bound but left too few functions in $\cA(D)$. Exact Fourier and kernel formulas reduced the task to finding $O(t)$ directions with sufficient relative distance and exponentially many independent functions in $\cA(D)$. ChatGPT first used rational points on a Garcia--Stichtenoth tower to obtain the distance bound. It then observed that matrices formed from multiplication in Riemann--Roch spaces change by rank one along evaluation directions, making their maximal minors affine on the corresponding lines. Proving the minors independent as polynomials and as functions completed the infinite construction in Sections~\ref{sec:relation-matrices} and~\ref{sec:ag-directions}.

In subsequent interactions, ChatGPT developed the two-dimensional proof for arbitrary $\lambda$ and proposed an extension to general dimension. A subsequent review of the literature showed that the higher-dimensional statement follows more directly from the product construction of Golowich~\cite{Gol21}. We therefore replaced the proposed extension with the deduction in Section~\ref{sec:dimension-lift}.

At this point, we humans became heavily involved. We checked and revised the proofs, reducing a draft of more than 50 pages to roughly 10 pages by simplifying arguments and removing unnecessary steps. We also identified constructions that ChatGPT had used without recognizing their appearance in prior work, and corrected the attribution.

The final argument underwent several rounds of AI review and independent checking by the human authors. During the formalization of parts of the argument in Lean, Codex identified further issues in the proofs and notation, which we corrected.

The following approaches considered by ChatGPT did not give the required construction.

\begin{itemize}
\item \textbf{Cyclic and voltage constructions.} Weight-three matrices from voltage constructions and cyclic group rings allowed exact Fourier and mod-$2$ calculations. In the examples studied, increasing the binary nullity either introduced an unacceptable real eigenvalue or still left the Cayley degree superpolynomial in $n$.
\item \textbf{Finite geometry.} Some small configurations from finite geometry and coding theory satisfied the real spectral bound and had nonzero binary nullity. Their kernels were too small, while the larger examples tested had real eigenvalues outside the required range.
\item \textbf{Locally testable codes and weight reduction.} Binary three-query locally testable codes have many weight-three checks and positive rate, but projectivizing their columns did not turn tester soundness into the required two-sided bound on the real incidence spectrum. A finite weight-reduction gadget worked, but no suitable outer tester was found.
\item \textbf{Covers and concatenation.} We considered two-lifts, symmetric covers, and concatenation with an outer code to extend small examples. The tested constructions either failed the spectral bound, added only kernel vectors constant on the fibers of the cover, or produced columns representing the same projective point.
\end{itemize}

%% file: appendix-ag-background.tex
\clearpage
\section{Background on algebraic geometry}
\label{app:ag-background}

This appendix recalls the facts about curves, divisors, and Riemann--Roch spaces used in Section~\ref{sec:ag-directions}. See~\cite[Chapter~1]{Sti09} for proofs and further background.

\paragraph{Curves, function fields, and rational points.} Let $X$ be a smooth projective geometrically integral curve over a finite field $F$, with function field $F(X)$. Smoothness excludes singular points: at each closed point $P$, a nonzero rational function has the form $f=z^m u$, where $z$ is a local parameter and $u$ is regular and nonzero at $P$; the integer $m=\operatorname{ord}_P(f)$ measures its zero or pole. Projectivity means that $X$ is a closed curve in projective space, so points at infinity are included in divisor calculations. Geometric integrality means that $X_{\overline F}$ is irreducible and reduced; reduced means that its local rings have no nonzero nilpotents. Consequently, products of nonzero rational functions remain nonzero over $\overline F$. An $F$-rational point is a degree-one place. If $f$ has no pole at $P$, evaluation $f(P)\in F$ is defined and satisfies $(fg)(P)=f(P)g(P)$.

\paragraph{Genus.} The genus $g$ of a smooth projective curve is the dimension of its space of regular differential forms. Over $\mathbb C$, it equals the number of handles of the associated compact surface. The projective line has genus zero, while an elliptic curve has genus one. We use genus through the Riemann--Roch bound $\ell(G)\ge\deg G+1-g$, where $\ell(G)=\dim_F L(G)$, with equality whenever $\deg G>2g-2$~\cite[Theorems~1.5.15 and~1.5.17]{Sti09}.

\paragraph{Divisors and Riemann--Roch spaces.} A divisor is a finite formal sum $G=\sum_P n_PP$ over the closed points of $X$, with integer coefficients and support $\operatorname{supp}G=\{P:n_P\ne0\}$. Its degree is $\deg G=\sum_P n_P[\kappa(P):F]$, where $\kappa(P)$ is the residue field of $P$.  When the support consists of rational points, this reduces to $\sum_P n_P$.  A nonzero rational function $f$ has a principal divisor $\operatorname{div}(f)=\sum_P\operatorname{ord}_P(f)P$, whose positive and negative parts record its zeros and poles and have the same degree.  Define
\[
 L(G)=\{0\}\cup
 \{f\in F(X)^\times:\operatorname{div}(f)+G\ge0\}.
\]
When $G=mP_\infty$, this is the $F$-vector space of functions with no poles away from $P_\infty$ and with pole order at most $m$ there.  If $G$ is effective and $0\ne f\in L(G)$, the total degree of its zeros outside $\operatorname{supp}G$ is at most $\deg G$.  We also use the elementary multiplication rule $L(A)L(B)\subseteq L(A+B)$.

\paragraph{Vanishing at specified points.} For a rational point $P\notin\operatorname{supp}G$, subtracting $P$ imposes a zero: $L(G-P)=\{f\in L(G):f(P)=0\}$. Thus, for distinct $P,R$ outside the support, a function in $L(G-R)\setminus L(G-P-R)$ vanishes at $R$ but not at $P$. In Section~\ref{sec:ag-directions}, Riemann--Roch shows that these two spaces differ in dimension by one, proving that evaluation distinguishes the selected points.

\paragraph{Evaluation as a direction.} Let $f_1,\ldots,f_t$ be a basis of $W=L(G)$. The functional $e_P\in W^\vee$, defined by $e_P(f)=f(P)$, has coordinates $(f_1(P),\ldots,f_t(P))$ in the dual basis. Its projective class $[e_P]$ is the direction used in the construction. A hyperplane in $W^\vee$ has the form $\{x:x(f)=0\}$ for some nonzero $f\in W$, so counting evaluation directions in that hyperplane amounts to counting zeros of $f$.

\paragraph{The algebra used for maximal minors.} The symmetric algebra $\operatorname{Sym}_F(W)$ is the polynomial ring on $W^\vee$: each $f\in W$ represents the linear form $x\mapsto x(f)$. The nonvanishing of products over $\overline F$ gives the $1$-genericity used in Section~\ref{sec:determinantal-affine-functions}. Polynomials and their induced functions must be distinguished: $z^{|F|}-z$ is a nonzero polynomial that vanishes on $F$.